\documentclass[12pt,reqno]{amsart}
\usepackage{amsfonts,amsmath,amssymb,amscd}
\usepackage{enumerate}
\usepackage{enumitem}
\usepackage{tikz}
\usepackage[colorlinks=true, linkcolor=blue, citecolor=blue, urlcolor=black]{hyperref}
\usepackage{url}

\newtheorem{theorem}{Theorem}[section]
\newtheorem{lemma}[theorem]{Lemma}
\newtheorem{corollary}[theorem]{Corollary}
\newtheorem{defn}[theorem]{Definition}
\newtheorem{remark}[theorem]{Remark}
\newtheorem{prop}[theorem]{Proposition}

\newcommand{\Q}{\mathbb{Q}}
\newcommand{\N}{\mathbb{N}}
\newcommand{\R}{\mathbb{R}}
\newcommand{\Z}{\mathbb{Z}}
\newcommand{\C}{\mathbb{C}}

\begin{document}
\title[Primitive prime divisors]{Primitive prime divisors in the forward orbit of a polynomial}

\address[Laishram, Yadav]{Stat-Math Unit, Indian Statistical Institute\\ 7 S. J. S. Sansanwal Marg, New Delhi, 110016, India}
\email[Laishram]{shanta@isid.ac.in}
\email[Yadav]{pkyadav914@gmail.com; yadavprabhakar096@gmail.com}

\author[Laishram]{Shanta Laishram}

\author[Rout]{Sudhansu S. Rout}
\address{Sudhansu Sekhar Rout, Department of Mathematics, National Institute of Technology Calicut, Kozhikode-673 601, India.}
\email{sudhansu@nitc.ac.in; lbs.sudhansu@gmail.com}

\author[Yadav]{Prabhakar Yadav}

\dedicatory{}
\begin{abstract}
	For the polynomial $f(z) \in \mathbb{Q}[z]$, we consider the Zsigmondy set $\mathcal{Z}(f,0)$ associated to the numerators of the sequence $\{f^n(0)\}_{n \geq 0}$.  In this paper, we provide an upper bound on the largest element of $\mathcal{Z}(f, 0)$. As an application, we show that the largest element of the set $\mathcal{Z}(f,0)$ is bounded above by $6$ when $f(z) = z^d + z^e +c \in \mathbb{Q}[z]$, with $d>e \geq 2$ and $|c|>2$. Furthermore, when $f(z) =z^d+c \in \mathbb{Q}[z]$ with $|f(0)| > 2^{\frac{d}{d-1}}$ and $d >2$, we also deduce a result of Krieger [Int. Math. Res. Not. IMRN, 23 (2013), pp. 5498-5525] as a consequence of our main result. 
\end{abstract}
\maketitle
\pagenumbering{arabic}
\pagestyle{headings}

\section{Introduction}

Let $\mathcal{U} = (u_1, u_2, \ldots)$ be a sequence of integers. We say that a term $u_n$ of the sequence $\mathcal{U}$ has a {\em primitive prime divisor} if there exists a prime $p$ such that $p\mid u_n$ but $p\nmid u_m$ for $1\leq m<n$. The set 
\[\mathcal{Z}(\mathcal{U}) = \{n\geq 1: u_n\; \mbox{does not have a primitive prime divisor}\}\] is called the Zsigmondy set of the integer sequence $\mathcal{U}$. The first question that one asks about the Zsigmondy set of a sequence is whether it is finite and this question has received a lot of attention. 
Bang \cite{Bang86} (for $b=1$) and  Zsigmondy \cite{Zsigmondy1892} proved that for any co-prime integers $a > b > 0,\ \mathcal{Z}(a^n -b^n)$ is a finite set.  Further, this result was extended to more general binary linear recurrence sequences. In fact, following the works of Carmichael \cite{Carmichael1913}, Schinzel \cite{Schinzel1974}, Stewart \cite{Stewart1977} and Voutier \cite{Voutier1998}, finally Bilu, Hanrot and Voutier \cite{BHV01} proved that $\mathcal{Z}(\mathcal{U})$ is a finite set for any non-trivial Lucas or Lehmer sequence of integers $\mathcal{U}$. 

Assuming that the Zsigmondy  sets under consideration are finite, it is natural to ask for explicit bounds for $\#\mathcal{Z}(\mathcal{U})$ and $\max \mathcal{Z}(\mathcal{U})$. For instance, Zsigmondy’s original theorem shows that for integers $a > b > 0$, we have $\max \mathcal{Z}(a^n -b^n)\leq 6$ and in particular $\max \mathcal{Z}(2^n - 1) = 6$. Also, the deep result of Bilu et al., \cite{BHV01} shows that $\max \mathcal{Z}(\mathcal{U}) \leq 30$ for any non-trivial Lucas or Lehmer sequence of integers $\mathcal{U}$.

 The questions related to Zsigmondy set have been also studied for non-linear recurrences sequences. For example, Silverman \cite{Silverman1988} first showed that Zsigmondy set is finite for a elliptic divisibility sequence, but gave no effective bound for the largest element in the Zsigmondy set. Later for some special elliptic curves, a uniform bound for the largest element in the Zsigmondy set is obtained (see \cite{EMW06, Ingram2007}). 
 
 Recently, several authors explored the subject of primitive divisors in recurrence sequences generated by the iteration of nonlinear polynomials and rational functions. For a set $S$ endowed with self map $f$ and for any $m \in \mathbb{N}\cup \{0\}$, we denote by $f^{m}$ the $m$-th iteration $f\circ  \cdots \circ f$ of $f$ with $f^{0}$ denoting the identity map on $S$. For $\alpha\in S$, we define the {\em (forward) orbit} by \[\mathcal{O}_{f}(\alpha) := \{f^{m}(\alpha)\mid  m \in \mathbb{N}\}.\] 
 
 \noindent We say that $\alpha$ is {\em preperiodic} if $f^{m+n}(\alpha) = f^m(\alpha)$ for some $n\geq 1$ and $m\geq 0$. Equivalenty, $\alpha$ is preperiodic if its orbit $\mathcal{O}_{f}(\alpha)$ is a finite set. A point that is not periodic, i.e., that has infinite $f$-orbit, is called a {\em wandering point}. If $f$ is a polynomial in $x$, and $(u_n)$ is given by $u_{n+1} = f(u_n)$ for $n\geq 1$, we say that $(u_n)$ is the sequence generated by $f$ starting at $u_1$, denoted by $(f, u_1)$. With this notion, the Zsigmondy set of the sequence $(f^n(\alpha))_{n\geq 1}$ is defined by 
 \[\mathcal{Z}(f, \alpha) = \{n\geq 1: f^n(\alpha)\; \mbox{does not have a primitive prime divisor}\}.\]
  
 \noindent In this direction, Rice \cite{Rice2007} first proved that for a monic polynomial $f(z) \in \Z[z]$, $f(z)\neq z^d$, if $0$ is a preperiodic of $f$ and $\alpha\in \Z$ has infinite orbit, then $\mathcal{Z}(f, \alpha)$ is finite.  Ingram and Silverman \cite{InSil09} later generalized this result to arbitrary rational maps over number fields. In fact, they proved that for any rational function $f(z)\in \Q(z)$ of degree $d\geq 2$ with $f(0) = 0$ and order of vanishing of $f$ at $z=0$ is not $d$, if $\alpha$ has infinite orbit, writing $f^n (\alpha) = \frac{A_n}{B_n} \in \Q$ in lowest terms, the Zsigmondy set $\mathcal{Z}((A_n)_{n\geq 0})$ is finite. Their proof, which relies on Roth's theorem, doesn't give an effective upper bound for $\max \mathcal{Z}((A_n)_{n\geq 0})$. 
 
 Hereafter, by $\mathcal{Z}(f,0)$ we denote the Zsigmondy set for the sequence defined by the numerators of $f^n(0)$. In \cite{DoHa12}, Doerksen and Haensch explicitly characterized the Zsigmondy set $\mathcal{Z}(f, 0)$ for the polynomial $f(z) = z^d + c$ of degree $d \geq 2$ with $c \in \Z$ and $\mathcal{O}_f(0)$ infinite. In fact, they proved that $\max \mathcal{Z}(f, 0) \leq 2$ if $c=\pm 1$ and $\mathcal{Z}(f, 0)$ is empty for all other $c\in \Z$.
 Krieger \cite{Krieger2013} considered the Zsigmondy set for such $f$ when $c \in \Q$ and showed that $\# \mathcal{Z}(f, 0) \leq 23$. 
 Recently, Ren \cite{Ren2021} further generalized the result of Krieger for more general polynomials which are not necessarily monic nor integer polynomial. The main result of \cite{Ren2021} asserts that for every polynomial $f\in \Q[x]$ of degree $d\geq 2$ with a critical point $u\in \Q$ there is a constant $M_f>0$, depending only on $f$ (and not on $c\in \Q$), such that $\# \mathcal{Z}(f_c, u)\leq M_f$ for every $c$ satisfying certain condition where $f_c(x) = f(x)+c$. For other related results in this direction, we refer to \cite{Shokri2022, Cheng2019, GNT13}. 
 
 In this paper, we study the question of the existence of an effective  bound on the largest element of $\mathcal{Z}(f, 0)$, and also finding a uniform bound on the largest element of the Zsigmondy set for some class of rational polynomials. To state our result, let
  
  \begin{equation}\label{e1-Zsig}
 f(z)= a_dz^d+ \cdots +a_1z+a_0, \mbox{with}\; a_i \in \Q, a_d \neq 0, \mbox{and}\; a_1=0
 \end{equation}
 and we define the following sets:
 \begin{align}\label{e2-Zsig}
 \begin{split}
		P^+&= \{ 0 \leq i \leq d : a_i=0 {\rm\  or\ } \mbox{sgn}(a_i)=\mbox{sgn}(a_d)\} \\
		P^-&= \{ 0 \leq i \leq d : a_i=0 {\rm\  or\ } \mbox{sgn}((-1)^ia_i)= \mbox{sgn}((-1)^da_d)\} \\
		N^{\pm}&=(P^{\pm})^c \quad (\mbox{complement of $P^{\pm}$ in \{0,1,2, \ldots,d\}})\\
		n^{\pm }&=\begin{cases}
			\max \{1,\{i: i \in N^{\pm}\}\} & \ {\rm if\ } N^{\pm} \neq \emptyset \\
			1 & \ {\rm if\ } N^{\pm}= \emptyset\\
		\end{cases}
		\end{split}
	\end{align} 
where $\mbox{sgn}(a)= a/|a|$ for any real number $a$. Let $z$ be such that $|z| \geq 1$ and satisfies
		\begin{align} \label{e3-Zsig}
			\sum_{n^+<i\leq d} |a_i||z|^{i-n^+} \geq \left( \sum_{i \in N^+} |a_i| \right) +1\quad \mbox{and}\;\; \sum_{n^-<i\leq d} |a_i||z|^{i-n^-} \geq \left( \sum_{i \in N^-} |a_i| \right) +1
		\end{align} 
		for both the sets $N^+$ and $N^-$. Then our main result is the following.
\begin{theorem} \label{thm1.1-Zsig}
	Let $f(z) \in \Q[z]$ be a polynomial of degree $d \geq 2$ as in \eqref{e1-Zsig} with $|a_0| \geq 1$. Let $\hat h_f$ be the associated canonical height. Further assume that $a_0$ satisfies inequality \eqref{e3-Zsig}. If $n \in \mathcal{Z}(f,0)$, then
		\begin{equation} \label{e4-Zsig}
			n \leq \frac{2}{\log d} \log \left( \frac{dC}{(d-1) \hat{h}_f(a_0)} \right) + 2
		\end{equation}
where $C \geq \sum_{v \in V_{K}} \log C_{v}$ and  $C_v$ is the associated constant in Remark \ref{rem2.9-Zsig}.
\end{theorem}

\noindent The following result of Krieger \cite[Proposition 5.3]{Krieger2013} can be easily seen as a corollary of Theorem \ref{thm1.1-Zsig} (see Subsection \ref{subs3.1-Zsig}).

\begin{corollary}\label{cor1.2-Zsig}
	Let $f(z)= z^d + c \in \Q[z]$ be a polynomial of degree $d \geq 3$ such that $c \in \Q \backslash \Z$ and $|c| > 2^{\frac{d}{d-1}}$. If $n \in \mathcal{Z}(f,0)$, then $\mathcal{Z}(f,0) = \emptyset$.
\end{corollary}

Next, we apply Theorem \ref{thm1.1-Zsig} to provide an explicit and uniform bound on the Zsigmondy set for the orbit of $0$ of polynomials $f(z)=z^d+z^e+c \in \Q[z]$ where $d>e\geq 2$. In particular, we prove the following result. 
	
\begin{theorem} \label{thm1.3-Zsig}
		Let $f(z)=z^d+z^{e}+c \in \Q[z]$ be a polynomial of degree $d >e\geq 2$ such that $c =\frac{a}{b} \in \Q$ and $|c| >2$. If $n \in \mathcal{Z}(f,0)$, then $n \leq 6$.
\end{theorem}
	
The proof of Theorem \ref{thm1.1-Zsig} is given in Section \ref{sec3-Zsig}. The method used in proving the theorem are inspired by the work of Krieger \cite{Krieger2013}. We would like to point out that the upper bound on $\mathcal{Z}(f,0)$ for polynomials of type \eqref{e1-Zsig} is enough for the upper bound on Zsigmondy set $\mathcal{Z}(g, u)$ of the sequence $(g^n(u)-u)_{n\geq 1}$ for any polynomial $g(z) \in \Q[z]$ with a critical point $u \in \Q$.  A simple calculation will yield that $g^n(u)-u=f^n(0)$, where $f(z) \in \Q[z]$ is the polynomial defined as $f(z) = g(z+u) - u$.

In Theorem \ref{thm1.3-Zsig}, we have taken $c$ to be rational which are not integers, because the case for an integer $c$ has been solved by Shokri \cite{Shokri2022}. In Section \ref{sec4-Zsig}, we give the proof of Theorem \ref{thm1.3-Zsig}. In Section \ref{sec5-Zsig}, we use the method similar to those used by Krieger \cite{Krieger2013} to obtain the upper bound of $\mathcal{Z}(f,0)$ when $|c|<2$ with certain assumptions on parity of $d$ and $e$.

Zsigmondy questions of this type also connect to broader problems in number theory and arithmetic dynamics. In 2013, assuming the $abc$-conjecture, Gratton, Nguyen and Tucker \cite{GNT13} proved the finiteness of Zsigmondy set for the numerator sequence of infinite orbit under rational iteration. Silverman and Voloch \cite{SiVo09} used Zsigmondy results of Ingram and Silverman \cite{InSil09} to prove that there is no dynamical Brauer-Manin obstruction for dimension $0$ subvarieties under morphisms $\phi$ between projective number field of degree at least $2$, whereas Faber and Voloch \cite{FaVo11}, have used the  Zsigmondy results of \cite{InSil09} in studying the nonarchimedean convergence of Newton's method.
 
\section{Auxiliary results}

Throughout the paper, $p$ will denote a prime number, and $v_p(\alpha)$ will denote the $p$-adic valuation of an integer $\alpha$. Let $f(z)$ be as in \eqref{e1-Zsig} and we write the $n^{th}-$iteration $f^n(0)$ in lowest form as
\begin{equation}\label{e5-Zsig}
	f^n(0)=\frac{A_n}{B_n},
\end{equation} 
where $B_n>0$ and co-prime to $A_n$. Recall the definition \[\mathcal{Z}(f,0):= \{n \in \mathbb{N} : A_n {\rm\ has\ no\ primitive\ prime\ divisor}\}.\]

At first  we will establish the rigid divisibility of the sequence $(A_n)_{n \geq 0}$ and state lemmas related to this property. Next, we define the concepts of local and global canonical heights and state results related to properties of canonical heights.

 \subsection{Rigid divisibility property:}
 A sequence $(u_n)_{n\geq 0}$ of integers is said to be a {\em rigid divisibility sequence} if for every prime $p$ the following properties hold:
 \begin{enumerate}
 \item If $v_p(u_n)>0$, then $v_p(u_{kn}) = v_p(u_n)$ for all $k\geq 1$.
 \item If $v_p(u_n)>0$ and $v_p(u_m)>0$, then $v_p(u_{\gcd(n,m)})>0$.
 \end{enumerate}

\noindent Let $p$ be a prime such that $p\mid A_{n_0}$ for some $n_0\in \N\cup\{0\}$. Set \(k(p) = \min\{n: p \mid A_n\}.\)

\begin{lemma}\label{lem2.1-Zsig}
Let $f(z)$ be as above in \eqref{e1-Zsig} and $A_n$ as in \eqref{e5-Zsig}. Suppose that $p$ is a prime that divides some element of the sequence $(A_n)_{n\geq 0}$. Then for every $n\in \N$, we have 
\[v_p (A_n)= \begin{cases}
v_p(A_{k(p)}) & \mbox{if}\;\; k(p)\mid n,\\
0 & \mbox{else.}
\end{cases}\]
\end{lemma}	

\begin{proof}
The proof follows from \cite[Lemma 2.3]{Krieger2013}.
\end{proof}

\noindent One can see that the following result is a consequence of Lemma \ref{lem2.1-Zsig}. 

\begin{lemma}\label{lem2.2-Zsig}
Let $f(z)$ be as above in \eqref{e1-Zsig} and $A_n$ as in \eqref{e5-Zsig}. Suppose that $n\in \N$ such that $A_n$ has no primitive prime divisor, that is, $n\in \mathcal{Z}(f,0)$. Then 
\begin{equation} \label{e6-Zsig}
A_n \mid \prod_{\substack{q \mid n \\ q \text{ prime}}} A_{\frac{n}{q}},
\end{equation}
where the product is taken over all distinct primes $q$ which divide $n$.
\end{lemma}
Taking absolute values and logarithms, we immediately have the following inequality, which will provide the starting point of all effective computations.
\begin{corollary}\label{cor2.3-Zsig}
Let $f(z)$ be as above in \eqref{e1-Zsig} and $A_n$ as in \eqref{e5-Zsig}. Suppose $A_n$ has no primitive prime divisor. Then 
\[\log |A_n| \leq \sum_{\substack{q \mid n \\ q: \text{prime}}} \log |A_{\frac{n}{q}}|.\]
\end{corollary}

\subsection{Canonical heights:}
Let $V_{\Q}$ be the set of places of $\Q$. For $p \in V_{\mathbb{Q}}$, we choose a normalized absolute value $|\cdot |_p$ in the following way. If $p = \infty$, then $|\cdot |_p$ is the ordinary absolute value on $\mathbb{Q}$, and if $p$ is prime, then the absolute value is the $p$-adic absolute value on $\mathbb{Q}$, with $|x|_p= p^{-v_p(x)}$ for any $x \in \Q^{\times}$.  These absolute values satisfy the product formula
\[\prod_{v \in V_{\Q}}|x|_{v} =1,\]
for any $x \in \Q^{\times}$. The standard (global) height function on $\Q$ is the function $h: \Q \to \R$ given by  $h(x) = \log \max \{|m|_{\infty}, |n|_{\infty}\}$, where $x =m/n$ in lowest terms. Equivalently, 
\begin{equation}\label{e7-Zsig}
h(x) = \sum_{v \in V_{\Q}} \log \max \{1, |x|_{v}\}, \quad \mbox{for any $x\in \Q^{\times}$}.
\end{equation}
This height function $h$ extends to the algebraic closure $\bar{\Q}$ of $\Q$ (see \cite[Section 3.1]{Silverman2007}). For any fixed polynomial $f(z) \in \Q[z]$ (or more generally, rational function) of degree $d\geq 2$, the canonical height function $\hat{h}_{f}:\bar{\Q} \to \R$ for $f$ is given by 
\begin{equation}\label{e8-Zsig}
\hat{h}_{f}(x): = \lim_{n\to \infty} \frac{h\left(f^n(x)\right)}{d^n}.
\end{equation}

\begin{lemma} [\cite{CaSil93, Silverman2007}]\label{heightprop}
	The canonical height function satisfy the following properties:
\begin{enumerate}[label=(\alph*)]
\item There is a constant $C$ depending only on $f$ such that 
$|\hat{h}_{f}(x)- h(x)| \leq C$ for every $x\in \bar{\Q}$.
\item $\hat{h}_{f}\left(f(x)\right)= d\cdot \hat{h}_{f}(x)$ for all $x\in \bar{\Q}$.
\end{enumerate}
\end{lemma}

\begin{defn}
For $v \in V_{\Q}$, let $\C_v$ denote the completion of an algebraic closure of $\Q$ with respect to $v$. The function $h_v: \C_v \to [0, \infty)$ given by
\[ h_{v}(x) := \log \max \{1, |x|_{v}\} \]
is called the \textit{standard local height} at $v$. Using this, \eqref{e7-Zsig} can be rewritten as
\[h(x) = \sum_{v \in V_{\Q}} h_{v}(x),\; \quad \mbox{for any $x\in \Q^{\times}$}.\]
\noindent If $f(z) \in \Q[z]$ is a polynomial of degree $d \geq 2$, the associated \textit{local canonical height}	is the function $\hat{h}_{v, f} (x) : \C_{v} \rightarrow [0, \infty) $ given by
\begin{equation}\label{e9-Zsig}
\hat{h}_{v, f} (x) = \lim_{n\to \infty} \frac{h_{v} \left(f^n(x)\right)}{d^n}.
\end{equation} 
\end{defn}

\noindent The local canonical heights provide a similar decomposition for $\hat{h}_{f}$, as follows.
 \begin{lemma}[\cite{CaSil93}, Theorem 2.3]
 Let $f(z)\in \Q[z]$ be a polynomial of degree $d\geq 2$. Then for all $x\in \Q$
 \[\hat{h}_{f}(x) = \sum_{v\in V_{\Q}}\hat{h}_{v, f} (x).\]
 \end{lemma}
 
 \noindent The following result of Benedetto \textit{et al.} \cite[Proposition 2.1]{BDJKRZ} is needed to estimate the constant $C$ in Lemma \ref{heightprop}(a).

 \begin{lemma}\label{lem2.7-Zsig}
 Let $K$ be a field with absolute value $v$, let $f(z) \in K[z]$ be a polynomial of degree $d \geq 2$, and let $\hat{h}_{v,f}$ be the associated local canonical height. Write $ f(z)= a_dz^d+ \cdots +a_1z+a_0= a_d (z-\alpha_1)\cdots (z-\alpha_d)$, with $a_i \in K$, $a_d \neq 0$, and $\alpha_i \in \mathbb{C}_{v}$. Let $A= \max \{|\alpha_i|_{v} : i=1,2,\ldots ,d\}$ and $B=|a_d|_{v}^{-1/d}$, and define real constant $C_{v}\geq 1$ by
\[C_v = \begin{cases}
\max \{1,A,B,|a_0|_{v},|a_1|_{v}, \ldots, |a_d|_{v}\} & \mbox{if}\;\; v\; \mbox{is nonarchimedean},\\
\max \{1,A+B,|a_0|_{v}+|a_1|_{v}+ \ldots+ |a_d|_{v}\} & \mbox{if}\; v\; \mbox{is archimedean}.
\end{cases}\]
 Then for all $ x \in \mathbb{C}_{v},$
		$$ \frac{-d \log C_{v}}{d-1} \leq \hat{h}_{v,f}(z)-h_{v}(z) \leq \frac{\log C_{v}}{d-1}. $$
 \end{lemma}
 \begin{remark} \label{rem2.8-Zsig}
		Note that from \cite[Remark 2.3]{BDJKRZ}, if $v$ is nonarchimedean then $A= \max\{|\alpha_i|_v\}$ can be directly computed from the coefficients of $f$. Specifically,
		$$ A=\max \left\{\left|\frac{a_j}{a_d}\right|^{1/(d-j)}_{v} : 0 \leq j \leq d-1\right\}. $$ On the other hand, if $v$ is archimedean then $A \leq \sum_{j=0}^{d-1} |a_j/a_d|^{1/(d-j)}_{v}$. Hence, the constant $C_{v}$ can be easily computed from the coefficients of $f$.
\end{remark}

{}

\begin{remark}\label{rem2.9-Zsig}
For $f(z) \in \mathbb{Q}[z]$, taking sum over all places $v$ of $\mathbb{Q}$, we obtain
		\begin{align} \label{e10-Zsig}
			\frac{-d C}{d-1} \leq \hat{h}_f(z) - h(z) \leq \frac{d C}{d-1}
		\end{align}
where $C$ is a constant satisfying $C\geq \sum_{v \in V_{K}} \log C_{v}$.
	\end{remark}
 
\section{Proof of Theorem \ref{thm1.1-Zsig}} \label{sec3-Zsig}

At first we will  establish the lower bound for $|f^k(z)|$. Precisely, we will prove that  $|f^k(z)| \geq 1$ for all $k \geq 1$. This observation will help us to use $h(f^k(0))$ and $\log |A_n|$ interchangeably.
\begin{prop} \label{prop3.1-Zsig}
		Let $f(z) \in \Q[z]$ be a polynomial of degree $d \geq 2$. Write $f(z)= a_dz^d+ \cdots +a_1z+a_0$, with $a_i \in \Q, a_d \neq 0$. Let $z$ be such that $|z| \geq 1$ and satisfies \eqref{e3-Zsig}.  Then for all $k \in \N$, we have $|f^k(z)| \geq |z|$.
\end{prop}
	\begin{proof} Let $z \in \R$ be such that $|z| \geq 1$ and satisfies \eqref{e3-Zsig}. If $z\geq 1$, then from \eqref{e2-Zsig} and \eqref{e3-Zsig}, we get
		\begin{align}\label{e11-Zsig}
		\begin{split}
			|f(z)| &= \left| \sum_{i \in P^+} a_i z^i +  \sum_{i \in N^+} a_i z^i\right|  \geq \left| \sum_{i \in P^+} a_i z^i \right|  - \left| \sum_{i \in N^+} a_i z^i\right| \\
			&\geq \sum_{i \in P^+} |a_i| |z|^i  - |z|^{n^+} \sum_{i \in N^+} |a_i|  \\
			&\geq  |z|^{n^+} \left( \sum_{n^+<i\leq d} |a_i| |z|^{i-n^+}  - \sum_{i \in N^+} |a_i|  \right) \geq |z|^{n^+} \geq |z|.
			\end{split}
		\end{align}
If $z\leq -1$, then proceeding as in \eqref{e11-Zsig} with the sets $P^-$ and $N^-$ and then using the assumption in \eqref{e3-Zsig}, we obtain 
		\[|f(z)| \geq |z|^{n^-} \geq |z|.\] 
Note that $|f(z)| \geq |z| \geq 1$. If $f(z)\geq 1$, then by the last inequality of \eqref{e11-Zsig}, 
		\begin{align*}
			|f^2(z)| &\geq |f(z)|^{n^+} \left( \sum_{n^+<i\leq d} |a_i| |f(z)|^{i-n^+}  - \sum_{i \in N^+} |a_i|  \right) \\
			&\geq |z|^{n^+} \left( \sum_{n^+<i\leq d} |a_i| |z|^{i-n^+}  - \sum_{i \in N^+} |a_i| \right) \geq |z|^{n^+} \geq |z|.
		\end{align*}
Similarly, if $f(z)\leq -1$, then $|f^2(z)| \geq |z|^{n^-} \geq |z| \geq 1$. Now Proposition \ref{prop3.1-Zsig} follows by induction on $k$.
	\end{proof}
	
\begin{proof}[Proof of Theorem \ref{thm1.1-Zsig}:]
Let $f(z) \in \Q[z]$ be a polynomial of degree $d \geq 2$ as in \eqref{e1-Zsig}. Then from \eqref{e5-Zsig}, we get
		$$ f^{k-1}(a_0)=f^{k}(0)=\frac{A_k}{B_k}. $$
		
Since $|a_0| \geq 1$ and $a_0$ satisfies the inequality \eqref{e3-Zsig}, by Proposition \ref{prop3.1-Zsig}, we have for all $k\geq 1,\; |f^{k-1}(a_0)| \geq |a_0| \geq 1$ and hence $|A_k|\geq B_k$. Thus,
	\begin{equation}\label{e12-Zsig}
	h(f^{k-1}(a_0)) = h(f^k(0))= \log (|A_k|), \quad \text{for all } k \geq 1.
	\end{equation}
Suppose that $ n \in \mathcal{Z}(f,0)$. By Corollary \ref{cor2.3-Zsig}, we have 
\begin{equation}\label{e13-Zsig}
\log |A_n| \leq \sum_{p \mid n} \log |A_{\frac{n}{p}}|
\end{equation}
 where the sum is taken over all distinct prime divisors $p$ of $n$. Then from \eqref{e12-Zsig} and \eqref{e13-Zsig}, we get
 \begin{equation}\label{e14-Zsig}
 h(f^{n-1}(a_0)) \leq \sum_{p \mid n} h(f^{\frac{n}{p}-1}(a_0)).
 \end{equation}
Let $C$ be the constant given in Remark \ref{rem2.9-Zsig}. From \eqref{e10-Zsig}, we have 
\begin{equation*}
\hat{h}_f(f^{n-1}(a_0))- \frac{d C}{d-1} \leq \sum_{p \mid n} \left(\hat{h}_f(f^{\frac{n}{p}-1}(a_0)) + \frac{d C}{d-1}\right).
\end{equation*}
Further, using functional relation in Lemma \ref{heightprop}(b), we get
\begin{equation*}
d^{n-1}\hat{h}_f(a_0) - \frac{d C}{d-1} \leq \hat{h}_f(a_0) \sum_{p \mid n} d^{\frac{n}{p}-1} + \omega(n)\frac{d C}{d-1}.
\end{equation*} 
where $\omega(n)$ denote the number of distinct prime factors of $n$. A simple calculation gives
\begin{equation}\label{e15-Zsig}
\frac{d^{n-1} - \omega(n) d^{\frac{n}{2}-1}}{\omega(n)+1} \leq \frac{d^{n-1} - \frac{1}{d} \sum_{p \mid n} d^{\frac{n}{p} } }{ \omega(n)+1} \leq \frac{dC}{(d-1) \hat{h}_f(a_0)}.
\end{equation}
Further, for $n \geq 4$, one can easily obtain that $2\omega(n)+1 \leq n-1 \leq 2^{\frac{n}{2}} -1 < d^{\frac{n}{2}}$. This yields
\begin{equation} \label{e-Zsig}
	2\omega(n) + 1 < d^{\frac{n}{2}}, \text{ for all } d \geq 3 \text{ and } n \geq 2.
\end{equation}
\noindent Hence, from \eqref{e15-Zsig}, we get

\[d^{\frac{n}{2}-1} \leq d^{\frac{n}{2}-1} \frac{d^{\frac{n}{2}}-\omega(n)}{\omega(n)+1} \leq \frac{dC}{(d-1) \hat{h}_f(a_0)}\] implying
\[n \leq \frac{2}{\log d} \log\left( \frac{dC}{(d-1) \hat{h}_f(a_0)} \right) +2.\] This completes the proof of Theorem \ref{thm1.1-Zsig}. 
\end{proof}

\begin{remark} \label{rem3.2-Zsig}
We use \eqref{e3-Zsig} and $|a_0| \geq 1$ to prove that $|f^n(a_0)| \geq 1$, so that we can replace $ \log |A_n| $ by $h(f^n(a_0))$. If we can ensure $|f^n(a_0)| \geq 1$, then without 
the assumptions \eqref{e3-Zsig} and $|a_0| \geq 1$ in Theorem \ref{thm1.1-Zsig}, we can directly use the bound on $n$ given there.
\end{remark}

\subsection{Proof of Corollary \ref{cor1.2-Zsig}:} \label{subs3.1-Zsig}
	From Remark \ref{rem2.9-Zsig}, we can easily see that $C= \log 2 + h(c)$ will work for the polynomial $f(z)$. Further, the remark following \cite[Lemma 6]{Ingram2009} gives the lower bound $\hat h_f(c) \geq \frac{1}{d} h(c)$ for $|c| > 2^{\frac{d}{d-1}}$. Using the above values in Theorem \ref{thm1.1-Zsig}, we deduce that $n \leq 5$. Further, with the simple observation $|f(z)| \geq |z|^{d-1}$ whenever $|z| >2$ and the rigid divisibility of numerator of the sequence $f^n(0)$, we deduce that $\mathcal{Z}(f,0) = \emptyset$. \qed

\section{Proof of Theorem \ref{thm1.3-Zsig}} \label{sec4-Zsig}

The following result of Benedetto \textit{et al.} \cite[Lemma 2.1]{BCHKW14} will be used to establish the lower bound of $\hat h_f(x)$ for $x \in \Q^{\times}$.
\begin{lemma} \label{lem4.1-Zsig}
		Let $f(z)=f_1(z)/f_2(z) \in \Q(z)$ where $f_1, f_2$ are relatively prime polynomials in $\Z[z]$ with degree $d:= \max \{\deg f_1,\deg f_2\} \geq 2$. Let $R=Res(f_1, f_2) \in \Z$ be the resultant of $f_1$ and $f_2$, and let 
\begin{equation}\label{e16-Zsig}
D:= \min_{t \in \R\cup\{\infty\}} \frac{\max \left\{|f_1(t)|,|f_2(t)|\right\}}{\max \{|t|^d,1\}}.
\end{equation}
		Then $D>0$, and for all $x\in \mathbb{P}^1(\Q)$ and all integers $i \geq 0$,
\begin{equation}\label{e17-Zsig}
 \hat{h}_{f}(x) \geq d^{-i} \left[ h(f^i(x)) - \frac{1}{d-1} \log \left(\frac{|R|}{D}\right) \right]. 
 \end{equation}
	\end{lemma}

\begin{proof}[Proof of Theorem \ref{thm1.3-Zsig}:]
Suppose that $f(z)=z^d+z^e+\frac{a}{b}$ where $c=\frac{a}{b} \in \Q$ for $a, b$ relatively prime integers and $|c| \geq 1$. We rewrite $f(z)$ as 
\[f(z)=\frac{f_1(z)}{f_2(z)}, \;\;\mbox{where $f_1(z)=bz^d+bz^e+a \in \Z[z]$ and $f_2(z)=b \in \Z[z]$}.\] The resultant of $f_1$ and $f_2$ is given by $R=\mbox{Res}(f_1, f_2)=b^d$. Now we will compute the lower bound for $D$ defined in \eqref{e16-Zsig}. Let $s:= \max \{2,|c|\}$. Then for $|t| \leq s$, we have $\max \{1,|t|^d\} \leq s^d$. Since $b \geq 1$, for $t \in \R \cup \{\infty\}$, 
\begin{equation}\label{e18-Zsig}
\frac{\max \{|f_1(t)|,|f_2(t)|\}}{\max \{|t|^d,1\}}\geq \frac{|f_2(t)|}{s^d} = \frac{b}{s^d} \geq \frac{1}{s^d}.
\end{equation}
Again for $|t| > s \geq 2$, we obtain

\begin{align}\label{e19-Zsig}
\begin{split}
\frac{\max \{|f_1(t)|,|f_2(t)|\}}{\max \{|t|^d,1\}} &\geq \frac{|f_1(t)|}{|t|^d} \geq b \left( 1-\frac{1}{|t|^{d-e}} - \frac{|c|}{|t|^d} \right) \\
&> \left( 1-\frac{1}{2}-\frac{1}{2^{d-1}} \right) \geq \frac{1}{4} \geq \frac{1}{s^d}
\end{split}
\end{align}
since $|a|>b > 1, |t|> s \geq 2$ and $d \geq 3$. From \eqref{e18-Zsig} and \eqref{e19-Zsig}, we get $D\geq \frac{1}{s^d}$. Substituting the lower bound for $D$ and the value of resultant $R$ in \eqref{e17-Zsig}, for any integer $i \geq 0$, $x \in \mathbb{P}^1(\Q)$ and $|c| \geq 2$, we obtain 
\begin{align*}
\hat{h}_{f}(x) &\geq d^{-i} \left[ h(f^i(x)) - \frac{1}{d-1} \log b^d + \frac{1}{d-1} \log \left( \frac{1}{s^d} \right)\right] \\
&= d^{-i} \left[ h(f^i(x)) - \frac{d}{d-1} \log |a| \right].
\end{align*}
Therefore, 
\begin{equation} \label{e20-Zsig}
\hat{h}_{f}(x) \geq d^{-i} \left[ h(f^i(x)) - \frac{d}{d-1} h(c)\right].
\end{equation}
For any $z \in \R$ with $|z| \geq |c| > 2$, 
	\begin{align*}
		|f(z)| &= |z^d + z^e + c| \geq |z|^d - |z|^e - |c| \geq |z|^{d-1} - |z|^e + |z|^{d-2} - |c| + |z|^{d-2} \geq |z|^{d-2} 
	\end{align*}
where the last inequality holds since $|z| \geq |c| >2$ and the fact that $d > e \geq 2$. This implies that $h(f(c)) \geq (d-2)h(c)$ and hence for $d\geq 5$, we have
\begin{equation}\label{e21-Zsig}
(d-1)h(f(c))-dh(c) \geq ((d-1)(d-2)-d) h(c) \geq dh(c).
\end{equation}
For $d=4$, the above inequality implies that 
\begin{equation}\label{e22-Zsig}
(d-1)h(f(c))-dh(c) \geq 2h(c).
\end{equation}
For $d=3$, one can notice that $\mbox{sgn}(c^d)=\mbox{sgn}(c)$, and hence  \[|f(c)| \geq |c|^3 - |c|^2 \geq |c|^2.\] Using similar argument as in \eqref{e21-Zsig}, for $d=3$ we obtain that 
 \begin{equation}\label{e23-Zsig}
 (d-1)h(f(c))-dh(c) \geq h(c).
 \end{equation} 
 Setting $i=1$ and $|c|>2$ in \eqref{e20-Zsig}, then from \eqref{e21-Zsig}, \eqref{e22-Zsig} and \eqref{e23-Zsig}, we get the lower bound for $\hat{h}_f(c)$ 
	\begin{equation*}
		(d-1)\hat{h}_f(c) \geq \begin{cases}
			h(c) & {\rm if\ } d \geq 5, \\
			\frac{1}{3}h(c) & {\rm if\ } d=4,3.
		\end{cases}
	\end{equation*}
Using the computations from Remark \ref{rem2.8-Zsig} and Lemma \ref{lem2.7-Zsig}, we obtain that 
\begin{equation} \label{e24-Zsig}
C_{v} = \begin{cases}
			|b^{-1}|_{v} & {\rm if\ } v \mbox{ is nonarchimedean}, \\
			2 + |c| & {\rm if\ } v \mbox{ is archimedean}.
		\end{cases}
\end{equation}
Hence, the constant $C$ in Theorem \ref{thm1.1-Zsig} for $f(z)= z^d+z^e+c$ can be taken as
 \begin{equation}\label{e25-Zsig}
 C=\log 2 + h(c).
 \end{equation} 
We would like to point out that \eqref{e3-Zsig} is satisfied for $d > e+1 \geq 3$. For the case $d=e+1$, recall that \eqref{e3-Zsig} was only needed to show that $|f^n(c)| \geq 1$ for all $n \in \N$, which can easily established using the relation $|f(z)| \geq |z|^{d-2}$, for all $|z| \geq |c| >2$. So, for $d\geq 3$ and $|c|>2$,  if $\ n \in \mathcal{Z}(f,0)$, then from \eqref{e4-Zsig}, we infer that
	\begin{align*}
		n &\leq \frac{2}{\log d} \log\left( \frac{dC}{(d-1) \hat{h}_f(c)} \right)+2  \leq \frac{2}{\log d} \log\left( \frac{3d(\log 2+ h(c))}{h(c)} \right)+2 \\
		&\leq \frac{2}{\log d} \log\left( 3d\left( \frac{\log 2}{\log 5}+1 \right) \right)+2 < 7,
	\end{align*}
	where the last inequality follows from the fact that $h(c) \geq \log 5$ which is clearly true as $|c|>2$ and $c \in \Q \backslash \Z$. This completes the proof of Theorem \ref{thm1.3-Zsig}.
\end{proof}	

\section{Few cases on $|c| <2$} \label{sec5-Zsig}

In this section, we establish the upper bound on Zsigmondy set of polynomial $f(z) = z^d + z^e +c \in \Q[z]$ when $|c|<2$. In this, we are only able to establish the upper bound on $\mathcal{Z}(f,0)$ under certain cases, that is, 
\begin{enumerate}[label=(\alph*)]
	\item $c \in (0,2)$,
	\item $c \in (-1,0)$ and $d$ is odd,
	\item $c \in (-2,-1)$ and either $d$ is odd or $e$ is even.
\end{enumerate}
The case (a) is proved in Propositions \ref{prop5.1-Zsig} and \ref{prop5.2-Zsig}, the case (b) and (c) are proved in Propositions \ref{prop5.3-Zsig} and \ref{prop5.4-Zsig}, respectively. 
For the remaining cases, one can use the ideas of Ren \cite{Ren2021} to bound the cardinality of $ \mathcal{Z}(f,0)$.
\begin{prop}\label{prop5.1-Zsig}
Let $f(z)=z^d+z^e+c \in \Q[z]$ be a polynomial of degree $d \geq 3$ with $c \in \Q$ and $1<c<2$. Then $n \leq 7$, whenever $n \in \mathcal{Z}(f,0)$.
\end{prop}	
\begin{proof}
We will proceed as in the proof of Theorem \ref{thm1.3-Zsig}. In this case, we have $s \leq 2c$. Hence, for any integer $i \geq 0$, $x \in \Q$ and $1<c<2$, we have
\begin{align*}
\hat{h}_{f}(x) &\geq d^{-i} \left[ h(f^i(x)) - \frac{1}{d-1} \log |b^d| + \frac{1}{d-1} \log \left( \frac{1}{s^d} \right)\right]\\
& \geq d^{-i} \left[ h(f^i(x)) - \frac{d}{d-1} \log (2a) \right]. 
\end{align*}
Thus, 
\begin{equation}\label{e26-Zsig}
\hat{h}_{f}(x) \geq d^{-i} \left[ h(f^i(x)) - \frac{d}{d-1} h(2c) \right].
\end{equation}
For $1<c<2$, we get 
	\begin{align*}
		f(c) = c^d+c^e+c = c(c^{d-1}+c^{e-1}+1) > 2c >2 
	\end{align*}
and this implies 
\begin{equation}\label{e27-Zsig}
h(f^2(c)) \geq dh(f(c)) \geq dh(2c).
\end{equation}
Multiplying  $(d-1)$ on both sides of \eqref{e27-Zsig} and then simplifying, we get
\begin{equation}\label{e28-Zsig}
(d-1)h(f^2(c))-dh(2c) \geq (d(d-1)-d) h(2c) \geq d(d-2)h(2c).
\end{equation}
Then from \eqref{e26-Zsig} and \eqref{e28-Zsig}, we deduce that
\begin{equation*}
(d-1)\hat{h}_f(c) \geq d^{-1}(d-2)h(2c).
\end{equation*}
Note that $f^n(c) > 1$ as $c>1$. Also, the constant $C$ in Theorem \ref{thm1.1-Zsig} for $f(z)= z^d+z^e+c$ can be taken as $C=\log 2+h(c)$. If $\ n \in \mathcal{Z}(f,0)$, then from \eqref{e4-Zsig}, we have
	\begin{align*}
		&n \leq \frac{2}{\log d} \log\left( \frac{dC}{(d-1) \hat{h}_f(c)} \right)+2 \\
		& \leq \frac{2}{\log d} \log\left( \frac{d^2(\log 2+ h(c))}{(d-2)h(2c)} \right)+2 \\
		&\leq \frac{2}{\log d} \log \left( \frac{2d^2}{d-2} \right)+2 \leq 2 \frac{\log 6d}{\log d} +2 < 8.
	\end{align*}
\noindent This completes the proof.
\end{proof}

\begin{prop}\label{prop5.2-Zsig}
Let $f(z)=z^d+z^e+c \in \Q[z]$ be a polynomial of degree $d >e\geq 2$ with $c \in \Q $ and $0<c<1$. Then $\mathcal{Z}(f,0) = \emptyset$.
\end{prop}	

\begin{proof}
Observe that for any $z>0$, $f(z)\geq z$. Using this observation and induction on $n$, we can conclude that $|f^n(c)| \geq c$ for all $n \geq 1$ since $c>0$. Also, 
\[|f(c)| = c^d+c^{e}+c \leq c(c^2+c+1)= \alpha c\] where $\alpha=c^2+c+1$. Clearly $1< \alpha <3.$ A simple induction will imply that 
\[|f^n(c)| \leq \alpha^{\frac{d^n-1}{d-1}} c.\] 
Hence, for $0<c<1$, we obtain that
\begin{equation}\label{e29-Zsig}
c \leq |f^n(0)| \leq \alpha^{\frac{d^{n-1}-1}{d-1}} c.
\end{equation}
From \eqref{e5-Zsig}, we have $f^n(0) = \frac{A_n}{B_n}$ with $B_n = b^{d^{n-1}}$. If $n\in \mathcal{Z}(f, 0)$, then from Corollary \ref{cor2.3-Zsig}, we obtain
\[\log |f^{n}(0)|+ d^{n-1} \log b \leq \sum_{q \mid n} \left( \log |f^{\frac{n}{q}}(0)|+d^{\frac{n}{q}-1} \log b \right)\] 
where the sum on the right is taken over distinct primes $q$ dividing $n$. Multiplying $d$ and using \eqref{e29-Zsig}, we have 
\[d\log c + d^n \log b \leq  \sum_{q \mid n} \left[ \frac{d^{\frac{n}{q}}-d}{d-1} \log \alpha  + d \log c+  d^{\frac{n}{q}} \log b\right],\]
rearranging above inequality, we get
\[d(1-\omega (n)) \log c + [d^n -s_d(n)]\log b \leq  \frac{\log \alpha}{d-1}  [s_d(n) - d \omega(n)],\]
where $s_d(n):= \sum_{q\mid n}d^{\frac{n}{q}}$. Since $d(1-\omega (n)) \log c$ is always non-negative, we have
\[[d^n -s_d(n)]\log b \leq  \frac{\log \alpha}{d-1}  [s_d(n) - d \omega(n)]\leq \frac{\log \alpha}{d-1}  s_d(n).\]

\noindent As $\alpha < 3 < b^2 $, we get
\[d^n-s_d(n)\leq \frac{2}{d-1}  s_d(n)\] and hence using $s_d(n)\leq d^{n/2}\omega(n)$ and \eqref{e-Zsig}, we get
\[d^{n} \leq \frac{d+1}{d-1} s_d(n) \leq 2 d^{n/2} \omega(n) < d^n,\]
for any $d \geq 3$ and $n \geq 2$. This is a contradiction.
\end{proof}
{}

\begin{prop}\label{prop5.3-Zsig}
Let $f(z)=z^d+z^e+c \in \Q[z]$ be a polynomial of odd degree $d >e\geq 2$ with $c \in \Q $. Suppose $-1<c<0$, then $\mathcal{Z}(f,0) = \emptyset$.
\end{prop}	

\begin{proof}
{\sc Case I:} ($e$ is odd). In this case,  we have the following inequality
\begin{equation*}
|c| \leq |f^n(0)| \leq | \alpha|^{\frac{d^{n-1}-1}{d-1}} |c|,
\end{equation*}
where $\alpha = c^2+|c|+1$. If $n\in \mathcal{Z}(f,0)$. Proceeding as in the proof of Proposition \ref{prop5.2-Zsig}, we get a contradiction.

\noindent {\sc Case II:} ($e$ is even). Since $-1<c<0,\ d$ is odd and $e$ is even, we have $c^d+c^{e}$ is positive and has absolute value less than $|c|$. So, we conclude that $|f(c)| = |c^d+c^{e}+c| \leq |c|$ and $f(c) < 0$. Suppose that $|f^n(c)| \leq |c| < 1$ and $f^n(c) <0$. Then
\[|f^{n+1}(c)|= |f(f^{n}(c))| = |(f^{n}(c))^d+ (f^{n}(c))^e+c|\leq |c|.\]
As $-1 < c < f^n(0) < 0$, clearly we have $f^{n+1}(c)<0$. Thus induction hypothesis yields
\begin{equation*}
-1<c \leq f^n(c) <0, \quad \text{for all } n \in \N.
\end{equation*}
Note that $d-e$ is odd. Thus, from above equation, we get 
\[|1+(f^n(c))^{d-e}|<1 \ \mbox{ and } \ |f^n(c)| \leq |c| \ \mbox{ for all } n \geq 0.\]
Now for $n\geq 1$, we have
\begin{align*}
|f^n(c)| &\geq |c|-|f^{n-1}(c)|^{e} |1+(f^{n-1}(c))^{d-e}| \\
&> |c|-|c|^{e}\geq |c|(1-|c|^{e-1}).
\end{align*}
Thus, 
\begin{equation*}
|c|(1-|c|^{e-1})\leq |f^n(c)| \leq |c|.
\end{equation*}
For $-1 < c = \frac{a}{b} < 0$, note that $|c|^{e-1} \leq |c|=\frac{|a|}{b}$. So, 
\[ (1-|c|^{e-1}) \geq \frac{b-|a|}{b} \geq b^{-1}.\]
If $n\in \mathcal{Z}(f, 0)$, then from Corollary \ref{cor2.3-Zsig}, we obtain
\[\log (|c|(1-|c|^{e-1}))+ d^{n-1} \log b \leq \omega(n) \log |c|+ \log b\sum_{q \mid n}d^{\frac{n}{q}-1}.\]
Multiplying by $d$ and rearranging, we have
\begin{align*}
[d^n -s_d(n)]\log b &\leq  d(\omega(n)-1)\log |c|  - d\log ((1-|c|^{e-1}))\\
&\leq - d\log ((1-|c|^{e-1}))\leq d\log b.
\end{align*}
By using $s_d(n)\leq d^{n/2}\omega(n)$ and \eqref{e-Zsig}, we obtain
\[d^n \leq s_d(n) + d \leq d^{n/2} \omega(n) + d \leq d^{n/2} (\omega(n)+1) < d^n\] 
for $d \geq 3$ and $n \geq 2$. This is a contradiction.
\end{proof}

{}

\begin{prop} \label{prop5.4-Zsig}
	Let $f(z) = z^d + z^e + c \in \Q[z]$ be a polynomial of degree $d>e \geq 2$ with $c \in \Q $. Suppose $-2 < c <-1$ and $d$ is odd or $e$ is even. Then $\mathcal{Z}(f,0) = \emptyset$.
\end{prop}
\begin{proof}
	Using simple inductive arguments, we obtain the upper bound 
	\[ |f^n(c)| \leq 3^{\frac{d^n-1}{d-1}} |c|^{d^{n}} \text{ for all } n \in \N. \]
	Now we consider different cases to get a lower bound for $|f^(c)|$. \\
	\noindent {\sc Case I:} ($d$ is odd). 
	As $c<-1$ and $d$ is odd, from
	\[ f(c)= c^d + c^e + c = -|c|( |c|^{d-1} \pm |c|^{e-1} + 1 ), \] we observe that $f(c)$ is negative and $|f(c)| \geq |c| >1$. Inductively using $f^{n-1}(c)<0$ and $|f^{n-1}(c)| \geq |c| >1$, we obtain from
	\[ f^{n}(c) = (f^{n-1}(c))^d + (f^{n-1}(c))^e + c < c <0, \]
	that $f^n(c)$ is negative and $|f^n(c)|  \geq |c| >1$ for all $n \in \N$.
	
	Thus, we may assume that $d$ is even, then $e$ is also even.
	
	\noindent {\sc Case II:} ($d$ and $e$ are even).
	Since $c<-1$ and both $d,e$ are even, from \\
	\( f(c) = c^d + c^e + c \geq c^d =|c|^d > 1 \) we observe that $f(c)$ is positive and $|f(c)| \geq |c|^d$. Inductively using $f^{n-1}(c) >0$ and $|f^{n-1}(c)| \geq |c|^{d^{n-1}} >1$, we obtain from
	\[ f^{n}(c) = (f^{n-1}(c))^d + (f^{n-1}(c))^e + c >  f^{n-1}(c)^{d} > |c|^{d^{n}},\]
	i.e., $f^n(c)$ is positive and $|f^n(c)| \geq |c|^{d^n}$ for all $n \in \N$.
	
	\noindent If $n\in \mathcal{Z}(f,0)$, then proceeding as in Proposition \ref{prop5.2-Zsig}, we get a contradiction for $d \geq 3$ and $n \geq 5$. For $n \leq 4$, simple manual check also contradicts the inequality arising from the Corollary \ref{cor2.3-Zsig}.
\end{proof}

{}

\section{Concluding Remark}
Let $f(z) = z^d + z^e +c \in \Q[z]$ be the polynomial of even degree $d$. If $c \in (-1,0)$ or $c\in (-2,-1)$ and $e$ is odd, then obtaining a non-trivial lower bound of $|f^n(c)|$ seems to be difficult. Hence, as a consequence, it is not easy to provide an explicit upper bound of $\mathcal{Z}(f,0)$ in such cases.

\section*{Acknowlegdments}
This work was started when the author P.Y. visited Department of Mathematics, NIT Calicut and he thanks the  institute for their hospitality. S.L. was supported by SERB-CRG grant CRG/2023/005564 while working on this project. S.S.R. was supported by grants from National Board for Higher Mathematics (NBHM), Sanction Order No: 14053 and from Anusandhan National Research Foundation (File No.:CRG/2022/000268) while working on this project.


\begin{thebibliography}{100}
   \bibitem{Bang86} A. S. Bang, \emph{Talteoretiske undersgelser}, Tidsskr. Math. {\bf 4}(5)  (1886), 70-80; 130-137.
	
	\bibitem{BDJKRZ} R. L. Benedetto, B. Dickman, S. Joseph, B. Krause, D. Rubin and X. Zhou, \emph{Computing points of small height for cubic polynomials}, Involve, {\bf 2} (2009), 37-64. 
	
	\bibitem{BCHKW14} R. L. Benedetto, R. Chen, T. Hyde, Y. Kovacheva and C. White, \emph{Small dynamical heights for quadratic polynomials and rational functions}, Expt. Math., {\bf 23} (2014), 433-447.
	
	\bibitem{BHV01} Y. Bilu, G. Hanrot and P.M. Voutier, \emph{Existence of Primitive Divisors of Lucas and
	Lehmer Number}, J. Reine Angew. Math, {\bf 539} (2001), 75-122. 
	
	 \bibitem{CaSil93} G. S. Call and J. H. Silverman, \emph{Canonical heights on varieties with morphisms}, Compositio Math. {\bf 89} (1993), 163-205.
	
	\bibitem{Carmichael1913} R. D. Carmichael, \emph{On the numerical factors of the arithmetic forms $\alpha^n\pm \beta^n$}, Ann. of Math. {\bf 15}(1/4)
	(1913), 49-70.
	
	\bibitem{Cheng2019} T. Cheng, \emph{Primitive prime divisors for weighted homogeneous polynomial}, Bull. Math. Soc. Sci. Math. Roumanie, {\bf 62} (2019), 173-182.
	
	 \bibitem{DoHa12} K. Doerksen and A. Haensch, \emph{Primitive prime divisors in zero orbits of polynomials},  Integers {\bf 12} (3)
	(2011), 465-473.
	
	\bibitem{EMW06} G. Everest, G. Mclaren, T. Ward, \emph{Prime divisors of elliptic divisibility sequences} J. Number Theory {\bf 118}  (2006), 71-89.
	
	\bibitem{Evertse1997} J. H. Evertse, \emph{The number of algebraic numbers of given degree approximating a given algebraic number}, Analytic Number Theory, Y. Motohashi (ed.), London Math. Soc. Lecture Notes Ser. {\bf 247}, (Cambridge University Press 1998), 53-83.
	
	\bibitem{FaVo11} X. Faber and F. Voloch, \emph{On the number of places of convergence for Newton's method over number fields}, J. Th\'eor. Nombres Bordeaux {\bf 23} (2) (2011), 387-401.
	
	\bibitem{GNT13} C. Gratton,  K. Nguyen and T. J. Tucker, \emph{ABC implies primitive prime divisors in arithmetic dynamic}, Bull. Lond. Math. Soc. {\bf 45} (6) (2013), 1194-1208.
	
	\bibitem{Ingram2007} P. Ingram, \emph{Elliptic divisibility sequences over certain curves}, J. Number Theory {\bf 123}(2)  (2007), 473-486.
	
	\bibitem{Ingram2009} P. Ingram, \emph{Lower bounds on the canonical height associated to the morphism $z^d+c$.}, Monatsh Math {\bf 157}(1)  (2009), 69-89.
	
	\bibitem{Ingram2012} P. Ingram, \emph{A finiteness result for post-critically finite polynomials}, Int. Math. Res. Not. {\bf3}  (2012), 524-543.
	
	 \bibitem{InSil09} P. Ingram and J. Silverman, \emph{Primitive divisors in arithmetic dynamics}, Math. Proc. Camb. Phil. Soc. {\bf 146} (2) (2009), 289-302.
	 
	 \bibitem{Ingram2019} P. Ingram, Canonical heights and preperiodic points for certain weighted homogeneous families of polynomials, Int. Math. Res. Not. {\bf 15} (2019), 4859-4879.  
	
  	\bibitem{Krieger2013} H. Krieger, \emph{Primitive prime divisors in the critical orbit of $z^d+c$}, Int. Math. Res. Not. IMRN, {\bf 23} (2013), 5498–-5525.
    
	\bibitem{Ren2021}R. Ren, \emph{Primitive prime divisors in the critical orbits of one-parameter families of rational polynomials}, Math. Proc. Camb. Philos. Soc., {\bf 171} (3) (2021), 569-584.

    \bibitem{Rice2007} B. Rice, \emph{Primitive prime divisors in polynomial arithmetic dynamics}, Integers {\bf 7}(1) (2007), A26, 1-16.
    
     \bibitem{Schinzel1962} A. Schinzel, \emph{The intrinsic divisors of Lehmer numbers in the case of negative discriminant}, Ark. Mat. {\bf 4} (1962), 413-416.
	
    \bibitem{Schinzel1974} A. Schinzel, \emph{Primitive divisors of the expression $a^n-b^n$ in algebraic number fields}, J. Reine Angew. Math. {\bf 268/269} (1974), 27-33.
    
    \bibitem{Shokri2022} K. M. Shokri, \emph{The Zsygmondy set for zero orbit of a rigid polynomial}, Khayyam J. Math. {\bf 8} (1) (2022),  115-119.
    
    \bibitem{Silverman1988}J. H. Silverman, \emph{Weiferich's criterion and the abc-conjecture}, J. Number Theory {\bf 30} (1988),  226-237.
    
    \bibitem{Silverman2007}J. H. Silverman, \emph{The arithmetic of dynamical systems}, Graduate Texts in Mathematics, vol. 241, Springer, New York, 2007.
    
    \bibitem{SiVo09}J. H. Silverman and F. Voloch, \emph{A local-global criterion for dynamics on {$\Bbb P^1$}}, Acta Arith. {\bf 137} (2009), 285-294.
    
    \bibitem{Stewart1977} C. L. Stewart, \emph{Primitive divisors of {L}ucas and {L}ehmer sequences}, Transcendence theory: Advances and Applications (A. Baker and D.W. Masser, eds.), Academic Press, New York, (1977), 79-92.
    
    \bibitem{Voutier1998} Paul M. Voutier, \emph{Primitive divisors of {L}ucas and {L}ehmer sequences. $III$}, Math. Proc. Camb. Philos. Soc. {\bf 123}(3) (1998), 407-419.
    
   \bibitem{Zsigmondy1892} K. Zsigmondy, \emph{Zur Theorie der Potenzreste}, Monatsh. Math. Phys. {\bf 3}(1) (1892), 265-284.
\end{thebibliography}
\end{document}